\documentclass[11pt]{amsart}
\usepackage{amsthm,amssymb, amsmath}
\usepackage{graphicx}
\usepackage[all]{xy}

\newtheorem{theorem}{Theorem}

\newtheorem{lemma}[theorem]{Lemma}
\newtheorem{proposition}[theorem]{Proposition}

\newtheorem{definition}[theorem]{Definition}
\newtheorem{remark}[theorem]{Remark}

\title[The local H\"older exponent]{The local H\"older exponent for the dimension of invariant subsets of the circle}

\author{Carlo Carminati}
\address{Dipartimento di Matematica, Universit\`a di Pisa, 
Largo Pontecorvo 5, 56127 Pisa, Italy}
\email{carminat@dm.unipi.it}

\author{Giulio Tiozzo}
\address{Yale University, 
10 Hillhouse Avenue, New Haven CT 06511, USA}
\email{giulio.tiozzo@yale.edu}

\begin{document}
 \maketitle

\begin{abstract}
We consider for each $t$ the set $K(t)$ of points of the circle
whose forward
orbit for the doubling map does not intersect $(0,t)$, 
and look at the dimension function $\eta(t) :=  \textup{H.dim }K(t)$.
 We prove that at every bifurcation parameter $t$, 
the local H\"older exponent of the dimension function 
equals the value of the function $\eta(t)$ itself. The same statement holds 
by replacing the doubling map with the map $g(x) := dx \mod 1$
for $d >2$.
\end{abstract}

\medskip

The theory of open dynamical systems, also known as dynamical systems ``with holes", 
was developed to model physical phenomena with escape of mass.
One of the simplest models which can be analyzed rigorously is the case 
of expanding maps of the circle $\mathbb{S} := \mathbb{R}/\mathbb{Z}$ where
the hole is an interval with zero on its boundary. 


More precisely, we shall fix an integer $d \geq 2$ once and for all, and consider $g(x) := dx \mod 1$ the map given by multiplication by $d$ on the circle $\mathbb{S}$. 
 For each $t \in [0, 1]$, let us define the set 
$$K(t) := \{ x \in \mathbb{R}/\mathbb{Z} \ : \ g^k(x) \notin (0, t) \quad \forall k \in \mathbb{N} \}$$
of elements whose forward orbit under $g$ does not intersect the interval $(0, t)$.
For each $t$, the set $K(t)$ is compact and forward-invariant for $g$.
One can see immediately that $K(0) = \mathbb{S}$ and $K(1) = \{0 \}$; moreover, $K(t)$ is a 
decreasing family of sets, in the sense that $s < t$ implies $K(s) \supseteq K(t)$.

We shall consider the \emph{dimension function} 
$$\eta(t) := \textup{H.dim }K(t)$$
which gives the Hausdorff dimension of the set $K(t)$ as a function of the parameter $t$.

The function $\eta(t)$ was introduced by Urba{\'n}ski \cite{Urb}, 
who proved that it is continuous, but not globally analytic. In fact, 
he showed that the dimension function is a ``devil's staircase", i.e. it is locally constant 
almost everywhere (see Figure \ref{F:HDU}).


In order to describe the finer analytical properties of $\eta(t)$, we shall call \emph{stable set} the set of parameters $t$ for which the
set-valued function 
$t \mapsto K(t)$ is locally constant at $t$, and the complement 
of the stable set will be called \emph{bifurcation set} and denoted by $\mathcal{U}$.
Clearly, the dimension function is locally constant on the stable set. 

On the other hand, we shall prove that on the bifurcation set 
the dimension function $\eta(t)$ has the following strong self-parametrizing property: 
$$\begin{array}{l}
\emph{at every bifurcation parameter, the local H\"older exponent of the dimension} \\
\emph{function equals the value of the function itself.}
\end{array}$$

To state the result precisely, let us define the \emph{local H\"older exponent} of a function $f : I \to \mathbb{R}$ at a point $x \in I$ as the limit
$$\alpha(f,x) := \liminf_{x' \to x}\frac{\log|f(x)-f(x')|}{\log|x-x'|}.$$
The main theorem is the following.

\begin{theorem} \label{T:local_h}
Let $d \geq 2$. Then, for each parameter $t$ in the bifurcation set, the local H\"older exponent of $\eta$ at $t$ equals $\eta(t)$, 
i.e.
\begin{equation} \label{E:localH} 
\alpha(\eta, t) = \eta(t).
\end{equation}
\end{theorem}

As a corollary, the dimension function is always locally H\"older continuous except at $t = \frac{d-1}{d}$
(where $\eta(t) = 0$) and becomes more and more regular as $t$ tends to $0$. It is differentiable at $t = 0$.

At $t = \frac{d-1}{d}$, the function is not H\"older continuous, and we shall show that its modulus of continuity is of order $\frac{\log \log (1/x) }{\log (1/x)}$ (see Proposition \ref{P:modulus}).

Moreover, the intervals where the function is constant correspond precisely to the connected components of the stable set, 
and these are characterized in terms of \emph{Lyndon words} (see below).  

\begin{figure}[h]
\fbox{\includegraphics[width=0.8\textwidth]{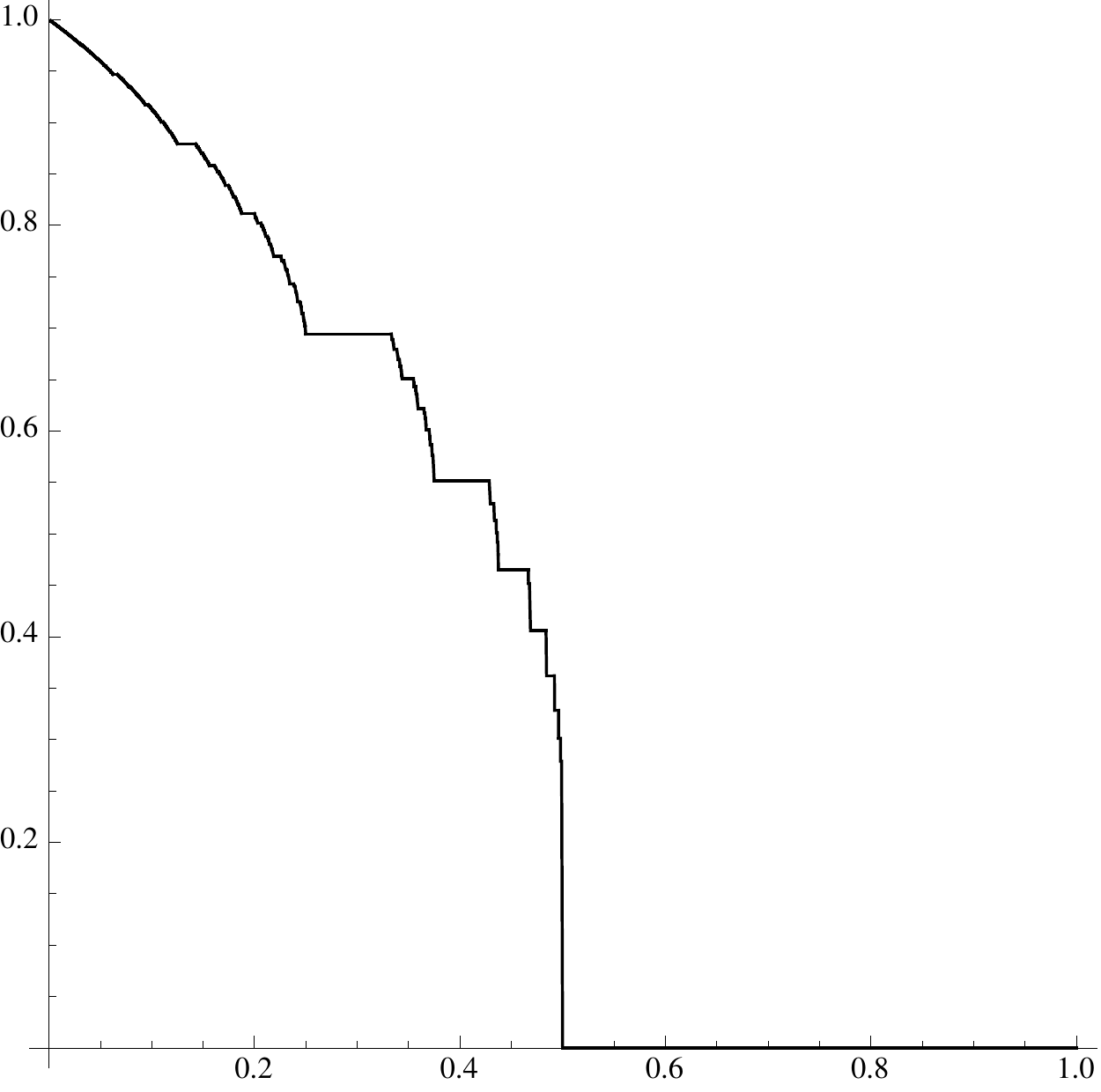}}
\caption{The dimension function $\eta(t)$ for $d = 2$.}
\label{F:HDU}
\end{figure}

\medskip
The dimension function $\eta$ is directly related to other quantities which have been 
widely studied. 
First of all, if we denote $M_n := \{x \in \mathbb{S} \ : \ g^k(x) \notin (0, t) \textup{ for }k = 0, \dots, n-1 \}$ the set 
of points which do not fall into the hole under the first $n$ iterates, one defines the \emph{escape rate} $\gamma$ to be 
$$\gamma := \lim_{n \to \infty} -\frac{1}{n} \log |M_n|.$$
The escape rate is directly related to $\eta$ by the formula 
$$\eta = 1 - \frac{\gamma}{\log d}.$$
In particular, the asymptotic behaviour of $\gamma$ in the ``small hole" case is quite well-understood 
(see among others \cite{kl} and \cite{dett}); this gives an asymptotic expansion of $\eta(t)$ as $t \to 0$. 

Moreover, if we denote by $h(t)$ the topological entropy of the restriction of the map $g$ to $K(t)$, 
one has the relation (see e.g. \cite{You})
$$h(t) = \eta(t) \cdot \log d.$$

The set $K(t)$ is also related to the problem of diophantine approximation (see \cite{Nilsson}): in fact,
if one considers the set 
$$F_t := \{ x \in \mathbb{S} \ : \ x - m/2^n \geq t/2^n \textup{ for all but finitely many }m,n \}$$
of points which are not well-approximable (in a suitable sense) by dyadic numbers, then one has for any $t \in [0,1]$
$$\eta(t) = \textup{H.dim }F_t.$$

\medskip

One important tool for the proof of Theorem \ref{T:local_h} is a formula, due to Urba{\'n}ski (\cite{Urb}, pg. 305) which allows to compute the value of $\eta(t)$ given the expansion in base $d$ of $t\in \mathcal{U}$.
More precisely, let $t \in \mathcal{U}$ be a bifurcation parameter, and let $t = .\epsilon_1\epsilon_2\dots$ its base-$d$ expansion which is not eventually $0$. 
Then $\eta = \eta(t)$ the Hausdorff dimension of $K(t)$ is given by 
\begin{equation}\label{eq:main}
\eta = -\frac{\log \lambda}{\log d}
\end{equation}
 where $\lambda$ is a root of the equation $P_t(\lambda) = 1$,
and $P_t(X)$ is the power series
\begin{equation} \label{E:PtX}
P_t(X) := \sum_{k = 1}^\infty (d-1-\epsilon_k)X^k.
\end{equation}

Let us stress that this formula is only valid for $t\in \mathcal{U}$, see Remark \ref{vedi-dopo}.
Indeed, the other main ingredient in our approach is an explicit characterization of the expansions in base $d$ of the elements of $\mathcal{U}$.

In particular, we will show (Proposition \ref{P:complement}) that the connected components of the complement of $\mathcal{U}$ are naturally labelled by \emph{Lyndon words}, i.e. finite words which are minimal for the lexicographic order among their cyclic permutations.
This also answers the question of Nilsson (\cite{Nilsson}, Sec. 6), who asks for a characterization of the plateaux of the dimension function for $d >2$; 
in the case $d =2$, our characterization is essentially equivalent to that given in \cite{Nilsson}.

In fact, using this description of $\mathcal{U}$ we will recover in a self-contained, elementary way the main results of \cite{Urb}, using combinatorics on words rather than thermodynamic formalism (see Remark \ref{final}). 

We are also aware of related work in progress by O. Bandtlow and H.H. Rugh, using a different approach based 
on thermodynamic formalism.
 
\medskip

Note that, without any reference to dynamics, one can ask whether there exist
monotone, continuous functions $f : [0,1] \to [0,1]$ with the 
property that, for each $t \in [0,1]$ either $f$ is locally constant at $t$, or the local H\"older exponent of $f$ at $t$ equals $f(t)$.
The functions $\eta(t)$ for varying $d$ provide infinitely many such examples (hence this property does not determine the function uniquely).

Moreover, functions with this property seem to appear also in relation to other families of dynamical systems. 

In particular, if one considers the function $h(\theta) := \frac{h_{top}(f_\theta)}{\log 2}$ which expresses the (normalized) topological 
entropy of a real quadratic polynomial $f_\theta$ as a function of its external angle $\theta$
(or equivalently, as a function of its kneading sequence) then it is also expected
that the local H\"older exponent of $h(t)$ at $t$ equals the value of the function $h(t)$ (see e.g. \cite{Gu} and \cite{IP}, Sec. 4).
Note that in this case the \emph{kneading series} of \cite{MT} plays the same role as the power series $P_t(X)$ in 
this paper.  
However, in this case the modulus of continuity at the smallest $t$ where $h(t) = 0$ (the Feigenbaum parameter) is of order $\frac{1}{\log(1/x)}$
(see e.g. \cite{Ti14}, Section 9.1).

Another more complicated, non-monotone case where the local H\"older exponent 
is at least conjectured to equal the value of the function at every point is given by
the (normalized) \emph{core entropy} function for quadratic polynomials introduced by W. Thurston (see e.g. \cite{Ti15}).

The underlying phenomenon, of which Theorem \ref{T:local_h} provides
a quantitative statement in a specific case, seems to be that systems with low entropy are less stable than systems with 
high entropy, in the sense that a small perturbation leads to a large variation in entropy. 
It would be of great interest to investigate to what extent is this phenomenon universal.

\section*{Word combinatorics and ordering}

Let $d \geq 2$ be fixed once and for all. We define the \emph{alphabet} as the set $\mathcal{A} := \{0, 1, \dots, d-1\}$, and (finite or infinite) sequences 
of elements of $\mathcal{A}$ will be called \emph{words}. If $S, T \in \mathcal{A}^n$ are finite words of equal length, 
we write $S < T$ to denote the lexicographical order; moreover, we shall 
extend the order to a partial order on the set of all finite words in the following way.
If $S = (a_1, \dots, a_n)$ and $T = (b_1, \dots, b_m)$ are finite words, we write $S \ll T$ 
(and read $S$ is \emph{strongly less} than $T$) if 
there exists an index $k \leq \min \{m, n\}$ such that 
$a_i = b_i$ for all $1 \leq i \leq k-1$ and $a_k < b_k$. 
For instance, $001$ is strongly less than $01$ but not strongly less than $00101$.

\begin{definition}
Let us define a finite word $S$ to be \emph{Lyndon} if it is strongly less than all its 
proper suffixes; that is, if for each decomposition $S = XY$
in two non-empty words one has 
$$S \ll Y.$$
\end{definition}

For instance, $011$ is Lyndon because $011 \ll 11$ and $011 \ll 1$, but $01101$ is not Lyndon, 
because $01$ is a suffix of $01101$ but $01101$ is not strongly less than $01$.

\begin{definition}
A rational number $r \in (0, 1)$ is called \emph{Lyndon} (for a given base $d$) if it admits a finite expansion $r = .\epsilon_1\dots\epsilon_m$ in base $d$
such that the word $S = \epsilon_1\dots\epsilon_m$ is Lyndon (note that such expansion is unique, because the Lyndon property implies $\epsilon_m \neq 0$).
\end{definition}

We shall denote $\mathbb{Q}_{(d)}$ the set of $d$-rational numbers contained in the interval $(0, 1)$, 
and $\mathbb{Q}_{Lyn}$ the set of Lyndon rational numbers in $(0, 1)$.

Finally, if $S = \epsilon_1\dots\epsilon_m$ is a finite word and $x \in [0, 1]$, we shall denote by 
$$S \cdot x := \sum_{i = 1}^m {\epsilon_i}d^{-i} + xd^{-m}$$
 the number whose expansion in base $d$ is $S$ followed by the expansion of $x$. The 
 affine map $x \mapsto S \cdot x$
 is an inverse to $g^m$ and is uniformly contracting with derivative $d^{-m}$.

Lyndon words appear in several contexts in combinatorics: for a reference, see e.g. \cite{Lot}, pg. 64.
Another equivalent definition given in the literature is that a word $S$ is Lyndon if it is the smallest among all its 
cyclic permutations, i.e. if one has 
$$S < YX$$
whenever $S = XY$ is a decomposition of $S$ in two non-empty words (for the equivalence, cf. Lemma \ref{L:dom}).

\section*{The bifurcation set}

Let us start by considering the function $t \mapsto K(t)$ as a function into sets.
We shall call a parameter $t \in [0, 1]$ \emph{stable} 
if the function $t \mapsto K(t)$ is locally constant at $t$: that is, 
if there exists $\epsilon > 0$ such that the equality
$$K(t') = K(t)$$
holds for each $t' \in [t-\epsilon, t + \epsilon]$.
We call such a set of parameters the \emph{stable set}.
A parameter which is not stable will be called a \emph{bifurcation parameter}, and 
the set of all bifurcation parameters will be called \emph{bifurcation set}
and denoted by $\mathcal{U}$.

The set $\mathcal{U}$ is closed with no interior, and has the following characterization:
$$\mathcal{U} = \{ t \in [0,1] \ : \ g^k(t) \notin (0,t) \ \forall k \geq 0 \}$$
(for a proof, see Lemma \ref{L:bif}.)
The main goal of this section is to characterize all connected components of the complement of $\mathcal{U}$; we shall 
see that they are naturally labeled by Lyndon rational numbers.

Let us define for each $d$-rational $r \in \mathbb{Q}_{(d)}$ the interval 
$$I_r := (.\epsilon_1\dots\epsilon_m, .\overline{\epsilon_1\dots\epsilon_m})$$
with left endpoint $r$ and right endpoint the rational number with periodic base-$d$ expansion $\epsilon_1\dots \epsilon_m$.
For instance, in the case $d =2$, if $r = 1/4= .01$, then $.\overline{01} = 1/3$ hence $I_{1/4} = (1/4, 1/3)$. Note also  $I_{1/2} = (1/2, 1)$.

\begin{proposition} \label{P:complement}
The connected components of the complement of $\mathcal{U}$ are parametrized 
by Lyndon rational numbers. Indeed, we have the identities
$$[0, 1] \setminus \mathcal{U} = \bigsqcup_{r \in \mathbb{Q}_{Lyn}} I_r = \bigcup_{r \in \mathbb{Q}_{(d)}} I_r.$$
\end{proposition}

The proposition will follow from the following lemmata.

\begin{lemma}\label{L:bif}
Let $t \in [0, 1)$. Then the following are equivalent:
\begin{enumerate}
\item
the element $t$ belongs to $K(t)$;
\item
$t$ is a bifurcation parameter.
\end{enumerate}
\end{lemma}

\begin{proof}[Proof of Lemma \ref{L:bif}]
If $t \in K(t)$, then for each $t' > t$ the element $t$ belongs to $K(t) \setminus K(t')$, proving (2).

If instead $t \notin K(t)$, then let $t' := \inf \{ x \in [t, 1] \ : \ x \in K(t) \} > t$ (the inequality is strict since $K(t)$ is closed).
We claim that $K(t') = K(t)$; indeed, if $x \in K(t)$ then for each $k$ we have $g^k(x) \in K(t) \subset [t', 1]$, 
hence $x \in K(t')$. Moreover, let us show that there exists $\epsilon > 0$ such that for $t' \in (t-\epsilon, t)$ we have 
$K(t') = K(t)$. If not, then there is a sequence of parameters $t_n \to t$, $t_n < t$ and a sequence of elements $x_n \in K(t_n) \setminus K(t)$.
By taking forward images of $x_n$, we then get a sequence $y_n = g^{k_n}(x_n) \in K(t_n) \cap [t_n, t]$: this implies that $y_n \to t$ and 
$$g^k(y_n) \in [t_n, 1]$$
for each $k$ and $n$: thus, since $g$ is continuous on $\mathbb{S}$, we have $g^k(t) \in [t, 1]$, which contradicts the fact that $t \notin K(t)$.
\end{proof}

\begin{lemma} \label{L:com1}
For each $d$-rational $r \in \mathbb{Q}_{(d)}$, the interval $I_r$ is contained in the stable set $[0,1]\setminus \mathcal{U}$.
\end{lemma}

\begin{proof}
Indeed, if $r = .\epsilon_1\dots\epsilon_m$, then the map $g^m$ is uniformly expanding of derivative $d^m$, 
it has $\overline{r} = .\overline{\epsilon_1\dots \epsilon_m}$ as its fixed point and maps $(r, \overline{r})$ onto $(0, \overline{r})$. Thus, if $x \in (r, \overline{r})$, 
then $|g^m(x) - \overline{r}| > |x - \overline{r}|$, hence $g^m(x) \in (0, x)$ and $x \notin \mathcal{U}$.
\end{proof}

\begin{lemma}\label{L:com2}
Let $x \notin \mathcal{U}$. Then $x$ belongs to some interval $I_r$ with $\partial I_r \subseteq \mathcal{U}$.
\end{lemma}

\begin{proof}
Let $x \notin \mathcal{U}$, and  $k \geq 1$ be the minimum value such that 
\begin{equation} \label{E:back}
g^k(x) \in (0, x).
\end{equation}
Let $x = .\epsilon_1\epsilon_2\dots$ be the non-degenerate expansion of $x$, denote $S_k := \epsilon_1\dots\epsilon_k$ 
its truncation and  write $r := S_k \cdot 0 = .\epsilon_1\dots \epsilon_k$. Note that the map $g^k$ is an orientation-preserving 
bijection from $J_k := S_k \cdot (0,1]$ onto $(0, 1]$ with derivative $d^k$, and $\overline{r}:= .\overline{\epsilon_1\dots\epsilon_k}$ is its fixed point. 
Now note that by construction $x$ belongs to $J_k$; moreover,  if $x \geq \overline{r}$, then $g^k(x) = d^k (x- \overline{r}) + \overline{r} \geq x$, contradicting 
eq. \eqref{E:back}. Thus, $x$ belongs to  $I_r := (r, \overline{r})$, proving the first part of the claim.

We claim moreover that for each $h = \{1, \dots, k-1\}$ we have 
\begin{equation}\label{E:incl}
g^h(J_k) \subseteq (\overline{r}, 1).
\end{equation}
which implies that both $r$ and $\overline{r}$ belong to $\mathcal{U}$, thus $\partial I_r \subseteq \mathcal{U}$ as required. To prove eq. \eqref{E:incl}, let us pick $y \in J_k$; 
if $y \geq \overline{r}$, then 
\begin{equation} \label{E:large}
g^h(y) = g^h(x)  + d^h(y-x) > x +  (y-x) = y \geq \overline{r}.
\end{equation}
Now, if there exists $y \in J_k \cap (0, \overline{r})$ 
such that $g^h(y)
< \overline{r}$, then by the intermediate value theorem there must exist 
$z \in J_k \cap (0, \overline{r})$ such that $g^h(z) = \overline{r}$, 
hence $g^k(z) = g^{k-h} 
(\overline{r}) \geq \overline{r}$ by the previous observation (eq. \eqref{E:large} with $k-h$ instead of $h$).
However, this is contradictory because $g^k(z) \in g^k(J_k \cap (0, \overline{r})) = (0, \overline{r})$.
\end{proof}

\begin{lemma} \label{L:dom}
Let $S = \epsilon_1\dots\epsilon_m \in \mathcal{A}^m$ be a word with $\epsilon_m \neq 0$, and $r = .\epsilon_1\dots\epsilon_m$
the associated $d$-rational.
Then the following are equivalent:
\begin{enumerate}
\item
$\partial I_r$ belongs to $\mathcal{U}$; 
\item 
$S$ is a Lyndon word.
\end{enumerate}
\end{lemma}

\begin{proof}
If $S$ is Lyndon, then for each $h \in \{1, \dots, m-1\}$ we have 
$$g^h(r) = .\epsilon_{h+1}\dots\epsilon_m \overline{0} > .\epsilon_1\dots \epsilon_m \overline{0}  = r$$
and similarly $g^h(\overline{r}) > \overline{r}$, thus the endpoints of $I_r$ belong to $\mathcal{U}$. Conversely, if $r \in \mathcal{U}$ then for each $h \in \{1, \dots, m-1\}$
$$ .\epsilon_{1}\dots\epsilon_m \overline{0} = r \leq g^h(r) = .\epsilon_{h+1}\dots\epsilon_m\overline{0}$$
hence $\epsilon_{1}\dots\epsilon_m \ll \epsilon_{h+1}\dots\epsilon_m$ unless $\epsilon_{h+1}\dots\epsilon_m$ is a prefix of $\epsilon_{1}\dots\epsilon_m$ 
and $\epsilon_{m-h+1}\dots\epsilon_m$ is all zeros, which is not possible since $\epsilon_m \neq 0$ by hypothesis.
\end{proof}

\begin{proof}[Proof of Proposition \ref{P:complement}]
It is easy to prove that the complement of $\mathcal{U}$ is open; namely, if $t \notin \mathcal{U}$, then 
by Lemma \ref{L:bif} (1) there exists $k \geq 0$ such that $g^k(t) \in (0, t)$, and such condition is open in $t$.
From Lemma \ref{L:com2}, Lemma \ref{L:dom} and Lemma \ref{L:com1} respectively
we have the chain of inclusions
$$[0, 1] \setminus \mathcal{U} \subseteq \bigcup_{\partial I_r \subseteq \mathcal{U}} I_r
\subseteq 
\bigcup_{r \in \mathbb{Q}_{(d)}} I_r \subseteq [0, 1] \setminus \mathcal{U}$$
thus equality must hold. Note also that two intervals $I_r$ whose endpoints lie in $\mathcal{U}$ may not overlap, 
hence their union must be disjoint. Moreover, by Lemma \ref{L:dom} the set of rationals $r$ for which $\partial I_r \subseteq \mathcal{U}$
coincides with the set $\mathbb{Q}_{Lyn}$ of Lyndon rationals, hence we get 
$$[0, 1] \setminus \mathcal{U} = \bigsqcup_{r\in \mathbb{Q}_{Lyn}} I_r
=
\bigcup_{r \in \mathbb{Q}_{(d)}} I_r.$$
As a consequence, the complement of $\mathcal{U}$ contains a left neighborhood of any  $d$-rational,
hence $\mathcal{U}$ has no interior.
\end{proof}


\section*{Structure and dimension of K(t)}

In this chapter we show that the set $K(t)$ has a countable Markov partition which we can easily describe, 
and can be used to compute the Hausdorff dimension of $K(t)$, thus giving an alternative proof of 
Urba{\'n}ski's formula (\cite{Urb}, pg. 305).
 
In order to state the result precisely, note that each real number $t \in (0, 1]$ admits exactly one expansion $t = .\epsilon_1\epsilon_2\dots$ 
in base $d$ such that the sequence $(\epsilon_n)_{n \in \mathbb{N}}$ is not eventually $0$. We shall call such expansion 
the \emph{non-degenerate} expansion of $t$. 

\begin{proposition} \label{T:main}
Let $d \geq 2$, and $t \in \mathcal{U}$ be a bifurcation parameter with non-degenerate base-$d$ expansion 
$t = .\epsilon_1\epsilon_2\dots$. 
Then $\eta = \eta(t)$ the Hausdorff dimension of $K(t)$ is given by 
$$\eta = -\frac{\log \lambda}{\log d}$$
 where $\lambda$ is a root of the equation 
\begin{equation}\label{uan2}
 P_t(\lambda) = 1
\end{equation}
and $P_t(X)$ is the power series
$$P_t(X) := \sum_{k = 1}^\infty (d-1-\epsilon_k)X^k.$$
\end{proposition}

Note that the series $P_t(X)$ always converges inside the unit disk and, by the intermediate value theorem, equation
\eqref{uan2} has exactly one root in the interval $(0, 1]$. Whenever $t$ has a purely periodic expansion of period $p$, the series $P_t(X)$ becomes a rational function and $\lambda = d^{-\eta(t)}$ is the root of a polynomial 
of degree $p$. 

As an example, in the case $d = 2$, if $t = .\overline{001} = 1/7$, then 
$$P_t(X) = X + X^2 + X^4 + X^5 + \dots = \frac{X + X^2}{1-X^3}$$
so $\lambda = 2^{-\eta(t)}$ is a root of $P_t(\lambda) = 1$, i.e. satisfies $\lambda^3+\lambda^2+\lambda-1 = 0$.

Let $t \in (0, 1]$, and $t = .\epsilon_1\epsilon_2\dots$ be 
its non-degenerate expansion in base $d$. For each $k \geq 1$ and $s \in \mathcal{A}$, define the word
$$S_{k,s}(t) := \epsilon_1\dots \epsilon_{k-1}s$$
and consider the set of words  
$$\Sigma(t) := \{ S_{k,s}(t)  \ : \ \epsilon_k <  s\}.$$
The following proposition characterizes precisely the elements 
which belong to $K(t)$ in terms of the set $\Sigma(t)$.

\begin{proposition}\label{P:struc}

Let $t \in \mathcal{U}$ be a bifurcation parameter. Then we have the identity
\begin{equation} \label{E:struc}
K(t) = \{t\} \cup \bigsqcup_{S \in \Sigma(t)} S \cdot K(t).
\end{equation}
That is, an element belongs to $K(t)$ if and only if its (non-degenerate) 
expansion in base $d$ is a concatenation of words in $\Sigma(t)$.
\end{proposition}

\begin{proof}
Let $x \in K(t)$. Then by definition $x \in [t, 1]$, hence either $x = t$ or the expansion of $t$ starts with $S_{k,s}$, 
where $k$ is the first digit for which the expansions of $t$ and $x$ differ, and $s$ is the $k^{th}$ digit of $x$, 
which mst be larger than the $k^{th}$ digit of $t$. Hence, $x = S_{k,s} \cdot y$ with $y \in [0, 1]$, 
and since $K(t)$ is forward invariant then also $y = g^k(x)$ belongs to $K(t)$, so $x \in S_{k,s} \cdot K(t)$.

Conversely, let $x = S \cdot y$ with $S \in \Sigma(t)$ and $y \in K(t)$. We have to prove that 
$g^h(x) \in [t, 1]$  for each $h \geq 0$. 
If $h \geq k$, then $g^h(x) = g^{h-k}(y) \in [t, 1]$ and the claim is proven.
On the other hand, fix $h \in \{0, \dots, k-1\}$ and compare the expansions of 
$g^h(x)$ and $g^h(t)$.
Since the expansion of $g^h(x)$ begins with $.\epsilon_{h+1}\dots\epsilon_{k-1} s$ and 
the expansion of $g^h(t)$ begins with $.\epsilon_{h+1}\dots\epsilon_{k-1} \epsilon_{k}$ and $s > \epsilon_k$, then 
we have 
$$g^h(x) \in [g^h(t), 1] \subseteq [t, 1]$$
where in the last inequality we used that $t$ belongs to $\mathcal{U}$, and the claim is proven.
\end{proof}

\begin{proof}[Proof of Proposition \ref{T:main}]
Consider the set $\widetilde{K}(t) := \{ x \in K(t) \ : \ g^k(x) \neq t \quad \forall k \geq 0 \}$.
Since $K(t)$ and $\widetilde{K}(t)$ differ by a countable set of preimages of $t$, their Hausdorff dimension is the same; 
moreover, by Proposition \ref{P:struc} we have 
$$\widetilde{K}(t) = \bigsqcup_{S \in \Sigma(t)} S \cdot \widetilde{K}(t).$$
The set $\widetilde{K}(t)$ is thus the attractor of a countable iterated function system; 
each map $x \mapsto S_{k,s} \cdot x$ is an affine map of derivative $d^{-k}$, and 
moreover, by construction all the images $S_{k,s} \cdot \widetilde{K}(t)$ are disjoint and satisfy the open set condition (\cite{Hut}, \cite{Moran}), hence the Hausdorff dimension 
$\eta$ of $\widetilde{K}(t)$ is determined implicitely 
by the formula 
$$1 = \sum_{S_{k,s} \in \Sigma(t)} d^{-k \eta}$$ 
which, since by definition of $\Sigma(t)$ for each $k$ there are $d-1-\epsilon_k$ values of $s$, can also be written as 
$$1 = \sum_{k = 1}^{\infty} (d-1-\epsilon_k) d^{-k \eta}$$
thus taking $X = d^{-\eta}$ yields the claim. 
\end{proof}

\begin{remark} \label{vedi-dopo}
Note that the hypothesis $t \in \mathcal{U}$ in Proposition \ref{T:main} is essential. Actually, one can 
define for any $t \in [0,1]$ the function $\zeta(t) = -\log \lambda /\log d$, where $\lambda$ is the unique real, positive root 
of the equation $P_t(X) = 1$. Then the function $\zeta(t)$ is no longer continuous, but 
for any $t \in [0,1]$ one has the relation (see Figure \ref{F:zeta})
$$\eta(t) = \min\{ \zeta(s), \ 0 \leq s \leq t \}.$$ 
\end{remark}

\begin{figure} 
\fbox{\includegraphics{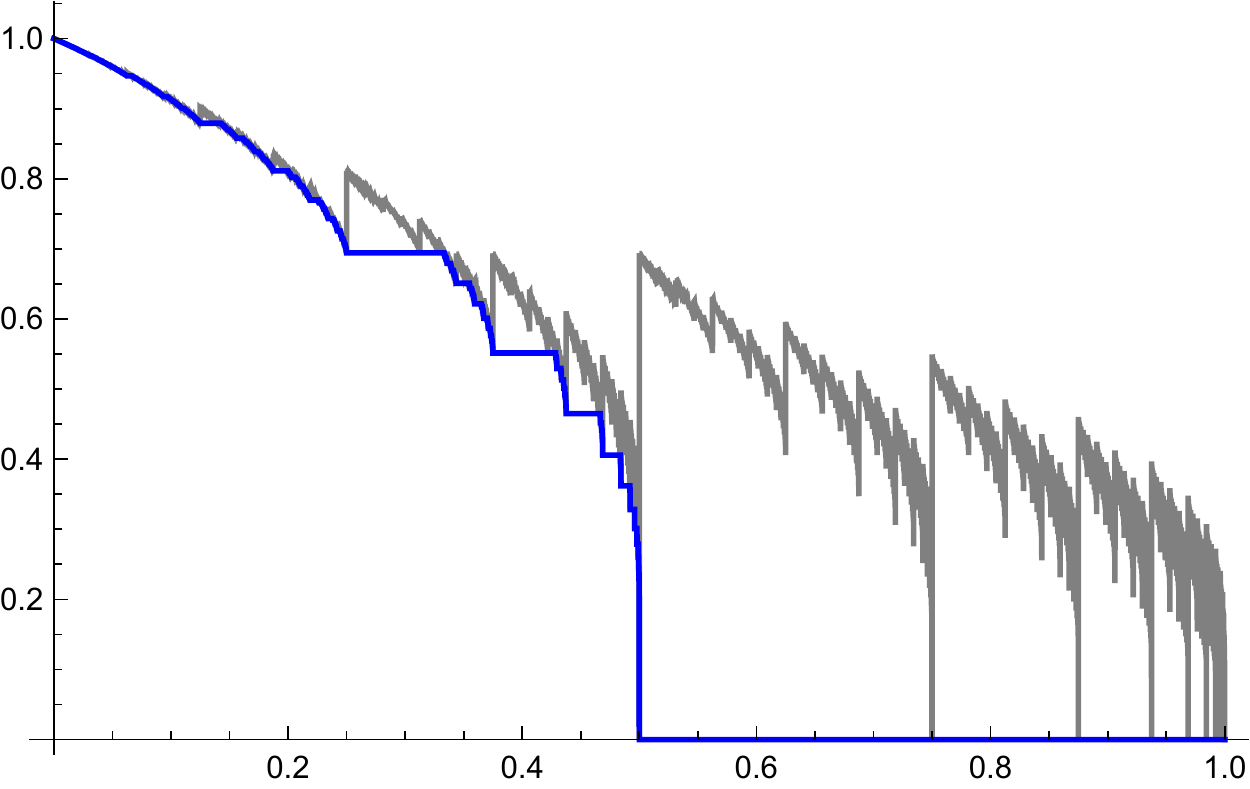}}
\caption{The functions $\zeta(t)$ and $\eta(t)$ for $d = 2$.} 
\label{F:zeta}
\end{figure}

\section*{The local H\"older exponent} \label{S:lhe}

Let us recall that a function $f : I \to \mathbb{R}$ on an interval $I$ is called \emph{H\"older continuous} of exponent $\alpha$
if there exists a constant $C > 0$ such that for each $x, y \in I$ one has 
$$|f(x) - f(y)| \leq C |x-y|^\alpha.$$
Given $t \in I$, 
we define the \emph{local H\"older exponent of $f$ at $t$} to be 
$$\alpha(f, t) := \liminf_{t' \to t} \frac{\log |f(t)-f(t')|}{\log |t-t'|}.$$

The goal of this section is to establish Theorem \ref{T:local_h}, namely that 
$$\alpha(\eta, t) = \eta(t)$$ 
for any $t \in \mathcal{U}$: let us start with some preliminary remarks.

Let us first note that for $t \in (\frac{d-1}{d}, 1]$ 
the set $K(t) = \{0\}$ is one point so $\eta(t) = 0$ and there is nothing to prove, so we shall assume $t \in [0, \frac{d-1}{d}]$.
In this case, let us note that $P_t(X) = sX + \sum_{k = 2}^\infty (d-1-\epsilon_k) X^k$ with $s \geq 1$,
hence the function $P_t(x)$ is strictly increasing on $[0, 1)$. Moreover, $P_t(0) = 0$ and $\lim_{x \to 1^-} P_t(x) > 1$ 
unless $t = \frac{d-1}{d}$ (in which case $P_t(x)$ is a polynomial), thus for each $t \in [0, \frac{d-1}{d}]$, 
the equation $P_t(x) =1$ has a unique solution $\lambda \in (0, 1]$, 
which we will denote $\lambda(t)$.
Note also that for each $x \in (0, 1)$ we have 
\begin{equation}\label{E:der}
1 \leq P'_t(x) \leq \frac{d}{(1-x)^2}
\end{equation}
hence $\lambda(t)$ is always a simple root of $P_t(X)$.

If $t \in [0,1]$, we shall denote $\epsilon_k(t)$ the $k^{th}$ digit of the non-degenerate expansion of $t$.
Moreover, if $t_1, t_2 \in [0,1]$, let us define $m(t_1, t_2)$ to be the length of the longest common prefix in the expansions of $t_1$ and $t_2$; 
namely, 
$$m(t_1, t_2) := \sup \{ k \geq 0 \ : \ \epsilon_h(t_1) = \epsilon_h(t_2) \quad \forall h \in \{ 1, \dots, k \} \}.$$

\begin{lemma}\label{L:gap}
For each $t_0 > 0$, there exists a constant $C_1 > 0$ such that for each $t_1, t_2 \in \mathcal{U} \cap [t_0, 1]$
one has 
\begin{equation}\label{E:step2}
C_1 d^{-m(t_1, t_2)} \leq |t_1 - t_2| \leq d^{-m(t_1, t_2)}.
\end{equation}
\end{lemma}

\begin{proof}
Let $m := m(t_1, t_2)$; the upper bound is given by $|t_1 - t_2| = d^{-m} |g^m(t_1) - g^m(t_2)| \leq d^{-m}$. To get the lower bound, 
first note that, since $K(t_0) \subseteq [t_0, 1]$ and $K(t_0)$ is forward invariant by $g$, we have 
$$K(t_0) \subseteq g^{-1}(K(t_0)) \subseteq g^{-1}([t_0, 1]) = \bigcup_{k=0}^{d-1} I_k$$
where  $I_k := \left[ \frac{t_0+ k}{d}, \frac{1+k}{d} \right]$. Now, by definition of $m$, the two points $u_1 := g^m(t_1)$ and $u_2 := g^m(t_2)$ 
belong to two different intervals $I_k$, thus 
$|t_1 - t_2| = d^{-m} |u_1 - u_2| \geq \frac{t_0}{d}$, which gives the lower bound with $C_1 := \frac{t_0}{d}$.
\end{proof}

We are now ready to prove the main theorem stated in the introduction.

\begin{proof}[Proof of Theorem \ref{T:local_h}]
Monotonicity of $\eta(t)$ is immediate from the definition, while continuity follows from Rouch\'e's theorem. Indeed, let $t \in \mathcal{U} \cap [0, \frac{d-1}{d}]$ and suppose $\lambda = \lambda(t) < 1$. 
Then, $\lambda(t)$ is a simple root of $P_t(X)$, and  $P_{t'}(X)$ converges uniformly on compact sets to $P_t(X)$ as $t' \to t$,
hence the root $\lambda(t')$ converges to $\lambda(t)$.
Suppose now $\lambda(t) = 1$, which implies $t = \frac{d-1}{d}$. Then
$P_t(X) - 1 = X - 1$ has no roots in any strip $\{ x + iy \ : \ 0 \leq x \leq 1-\delta, |y| \leq \delta \}$ 
for $0 < \delta < 1$, hence $\lambda(t') \to 1$ as $t' \to t$. Since $\eta(t)  = 
-\frac{\log \lambda(t)}{\log d}$, then continuity of $\lambda(t)$ implies continuity of $\eta(t)$.

Let us now estimate the modulus of continuity of $\eta(t)$. First note that, since $\eta(t) = -\frac{\log \lambda(t)}{\log d}$ and 
the function $h(x) := -\frac{\log x}{\log d}$ is bi-Lipschitz on $[1/d, 1]$, it is equivalent to prove the claim for $\lambda(t)$.
Let $t_1, t_2 \in \mathcal{U} \cap [0, \frac{d-1}{d}]$, and 
to simplify notation, we denote $\lambda_1 := \lambda(t_1)$, 
$\lambda_2 := \lambda(t_2)$, and also $P_1(X) := P_{t_1}(X)$, $P_2(X) := P_{t_2}(X)$
and suppose $\lambda_1, \lambda_2 < 1$.
 
Now, using that $P_2(\lambda_2) = P_1(\lambda_1) = 1$ and applying Lagrange's theorem we have
\begin{equation}\label{E:step1}
 P_1(\lambda_1) - P_2(\lambda_1) = P_2(\lambda_2) - P_2(\lambda_1) = P_2'(\xi) (\lambda_2 - \lambda_1)
\end{equation}
for some $\xi \in [\lambda_1, \lambda_2]$. 
On the other hand, by writing out the power series we get
\begin{equation}\label{E:step3}
P_1(\lambda_1) - P_2(\lambda_1) = \lambda_1^m R(t_1, t_2)
\end{equation}
 where $R(t_1, t_2) := 
 (\epsilon_m(t_2) - \epsilon_m(t_1)) + \sum_{j = m+1}^\infty (\epsilon_j(t_2) - \epsilon_j(t_1)) \lambda_1^{j-m}$
 and $m = m(t_1, t_2)$.
By comparing the two previous equations we get
$$|\lambda_1 - \lambda_2| =  \lambda_1^m \frac{|R(t_1, t_2)|}{P_2'(\xi)} \leq \lambda_1^m \frac{d}{1-\lambda_1}$$
hence by combining it with the upper bound for $|t_1-t_2|$ given by eq. \eqref{E:step2} we have the following upper bound for the modulus 
of continuity: for each $t \in \mathcal{U} \cap (0, \frac{d-1}{d})$, there exists $C_2 >0$ such that 
one has
$$|\lambda_1 - \lambda_2| \leq C_2 |t_1 -t_2|^{\frac{-\log \lambda_1}{\log d}} = C_2 |t_1 -t_2|^{\eta(t_1)}$$
for each $t_1, t_2 \in \mathcal{U}$ sufficiently close to $t$. 
Since $\lambda(t)$ is constant on the complement of $\mathcal{U}$, the above upper bound 
actually works for \emph{any} $t_1, t_2$ close to $t$, thus proving
$$\alpha(\eta, t) \geq \eta(t)$$ 
for each $t \in \mathcal{U}$.

For the lower bound, let us pick $t\in \mathcal{U}$.
Now, by Lemma \ref{L:approx} 
there exists a sequence $t_n \to t$ with $t_n \neq t$ such that either for each $k$ and each $n$ we have 
$\epsilon_k(t) \leq \epsilon_k(t_n)$, or we have the reverse inequality for each $k$ and each $n$; in both cases, 
 $R(t, t_n)$ is a power series in $\lambda(t)$ all of whose coefficients are integers and have the same sign, hence
$|R(t, t_n)| \geq 1$ and
$$|\lambda(t) - \lambda(t_n)| =  \lambda(t)^m \frac{|R(t, t_n)|}{P_2'(\xi)} \geq \lambda(t)^m \cdot \frac{(1-\xi)^2}{d} \geq C_3 |t-t_n|^{\eta(t)}$$
where $C_3 := \frac{(1-\lambda(t_1))^2}{d}$, 
proving the lower bound $\alpha(\eta, t) \leq \eta(t)$.
\end{proof}

\begin{lemma} \label{L:approx}
Let $t \in \mathcal{U}$. If $t$ is not a $d$-rational, then there exists a sequence $(t_n)$ of elements of $\mathcal{U}$
such that $t_n \to t$, $t_n > t$ for any $n$, and 
$$\epsilon_k(t) \leq \epsilon_k(t_n) \qquad \forall k,n.$$
If $t$ is a $d$-rational, then there exists a sequence $(t_n)$ of elements of $\mathcal{U}$
such that $t_n \to t$, $t_n < t$ for any $n$, and 
$$\epsilon_k(t_n) \leq \epsilon_k(t) \qquad \forall k,n.$$
\end{lemma}

\begin{proof}
Let $t \in \mathcal{U}$ not a $d$-rational, and let $t = .\epsilon_1\epsilon_2\dots$ be its (non-degenerate) expansion in base $d$. 
For each $n$, let us define 
$$t_n := .\epsilon_1\dots \epsilon_n (d-1)^\infty.$$ 
By construction $t_n > t$, $t_n \to t$, and $\epsilon_k(t_n) \geq \epsilon_k(t)$ 
for each $k$. We need to check that 
$t_n \in \mathcal{U}$. Let us consider $g^r(t_n)$ and compare it to $t_n$. 
If $r \geq n$, then $g^r(t_n) = 0 \notin (0, t_n)$ as required. If $r < n$ instead, then we have 
$g^r(t_n) = .\epsilon_{r+1}\dots \epsilon_n (d-1)^\infty$.
Since $t \in \mathcal{U}$, then $g^r(t) = .\epsilon_{r+1}\dots \epsilon_n \dots \geq t = .\epsilon_1\dots \epsilon_n \dots$, 
hence, if you set $S := \epsilon_1 \dots \epsilon_n$ and $S_0 := \epsilon_{r+1}\dots \epsilon_n$, either $S_0 \gg S$ or $S_0$ is a prefix of $S$. In the first case $g^r(t_n) \geq t_n$ as required; 
in the second case, $S_0 (d-1)^{n-r} \geq S$, so 
$g^r(t_n) = .S_0 (d-1)^\infty \geq .S (d-1)^\infty = t_n$, as required. 

Let us now deal with the case where $t$ is a $d$-rational, and let $t = .\epsilon_1 \dots \epsilon_k (d-1)^\infty$ be its non-degenerate 
expansion, which we can take so that $\epsilon_1 \neq d-1$ and $\epsilon_k \neq d-1$. We claim that the number $t_n$ with base-$d$ expansion
$$t_n = .\overline{\epsilon_1 \dots \epsilon_k (d-1)^n}$$
satisfies the claim. Clearly, $t_n < t$ and $t_n \to t$, while $\epsilon_k(t_n) \leq \epsilon_k(t)$ for any $k$. We need to prove that $t_n \in \mathcal{U}$.
Given $r$, consider $g^r(t_n)$: either the first digit of $g^r(t_n)$ is $(d-1)$, which implies $g^r(t_n) \notin (0, t_n)$ as $\epsilon_1 \neq d-1$, 
or $g^r(t_n)$ is of the form $g^r(t_n) = .\epsilon_{r+1} \dots \epsilon_k (d-1)^n \dots$. 
Then either $S_0 = \epsilon_{r+1} \dots \epsilon_k$ satisfies $S \ll S_0$, or $S_0$ is a prefix of $S$.
If $S_0$ is a prefix of $S$, then one can write $S = S_0 S_1$ where $S_1$ is some non-empty word, and either $S_1 \ll (d-1)^n$ or $S_1$ is of the form $(d-1)^a$ for some $a \geq 1$, which contradicts the fact that $\epsilon_k \neq d-1$.
\end{proof}

Note that an alternative way to define the local H\"older exponent of $f$ at $t$ is as
the supremum of all values $s$ for which 
$f$ is H\"older continuous of exponent $s$ on some neighborhood of $t$, i.e. as
$$\widetilde{\alpha}(f, t) := \sup \left\{ s \ : \  \lim_{\epsilon \to 0} \sup_{\stackrel{x, y \in B(t, \epsilon)}{x\neq y}} \frac{|f(x)-f(y)|}{|x-y|^s} < \infty \right\} $$
where $B(t, \epsilon)$ is the open ball of radius $\epsilon$ and center $t$.
Note that $\widetilde{\alpha}(f, t) > 0$ if and only if $f$ is locally H\"older continuous at $t$. 
While in general $\widetilde{\alpha}(f,t) \leq \alpha(f,t) $, the two quantities need not be the same; however, in our case the 
same argument as in the proof of 
Theorem \ref{T:local_h} shows that 
$$\widetilde{\alpha}(f,t) = \alpha(f,t) = \eta(t) \qquad \forall t \in \mathcal{U}.$$

Now we shall show that the function $\eta$ is not H\"older continuous at $t_*=1-1/d$ (which is the smallest $t$  such that $\eta(t)=0$). In fact the 
modulus of continuity of $\eta$ at $t_*$ is given by the function
$$\omega(x) :=\frac{\log\log(1/x)}{\log(1/x)}$$
as shown in the following Proposition.

\begin{proposition}\label{P:modulus}
We have the limit
$$\lim_{ t \to t^-_*} \frac{\eta(t)-\eta(t_*)}{\omega(t_*-t)}=1.$$
\end{proposition}

\begin{proof}
To begin we shall give a precise estimate of $\eta(t_n)$ where $t_n:=t_*-1/d^n$. It is easy to check that $t_n\in \mathcal{U}$ for all $n\geq 2$ and $P_{t_n}(X)= X+X^n$. In order to locate the unique positive solution $\lambda_n$
of the equation $P_{t_n}(X)=1$ let us observe that for any fixed $\alpha > 0$ one has
$$P_{t_n}\left(1-\alpha \frac{\log n}{n}\right)= 1-\alpha \frac{\log n}{n} + \frac{1}{n^\alpha}\left[1+O\left(\frac{\log^2 n}{n}\right)\right] \qquad \textup{as }n \to \infty.$$
Therefore, using the above formula for $\alpha=1$ we get that there is $n_0$ such that
 $P_{t_n}(1-\frac{\log n}{n})
 <1 \ \forall n\geq n_0$. On the other hand, for any $\alpha<1$ there is $n_1=n_1(\alpha)$ such that 
$P_{t_n}(1-\alpha \frac{\log n}{n})
 >1 \ \forall n\geq n_1$.
This means that as $n \to \infty$ we have $\lambda_n=1-  \frac{\log n}{n}[1+o(1)]$ and $$\eta(t_n)=-\frac{\log \lambda_n}{\log d}=\frac{\log n}{n\log d}[1+o(1)].$$
Recalling that $\log(t_*-t_n) = - n \log d$ we see that the modulus of continuity of $\eta$ at $t_*$ cannot be smaller than $\omega$, indeed:
$$\lim_{n \to +\infty} \frac{\eta(t_n)-\eta(t_*)}{\omega(t_*-t_n)}=1$$
On the other hand,  if $t_n\leq t \leq t_{n+1}$ then, using the monotonicity of $\eta(t)$ and $\omega(t)$ and the fact that $\frac{\omega(t_*-t_n)}{\omega(t_*-t_{n+1})} \to 1$, we get 
$$\lim_{ t \to t^-_*} \frac{\eta(t)-\eta(t_*)}{\omega(t_*-t)}=1.$$
 
\end{proof}

\begin{remark}\label{final}
Using the characterization of $\mathcal{U}$ we get also an elementary proof of the following result of Urba{\'n}ski (\cite{Urb}, Theorem 2):
$$\textup{H.dim }K(t)=\textup{H.dim }(\mathcal{U}\cap [t,1])\qquad \forall t \in [0,1].$$
\end{remark}

Indeed, if we denote by $\mathcal{P} \subset \mathcal{U}$ the set of $d$-rational elements of $\mathcal{U}$, it is easy to check that $\eta(\mathcal{P})$ is dense in $[0,1]$.
Therefore, since both the function $\eta(t)$ and $\tilde{\eta}(t):=\textup{H.dim }(\mathcal{U}\cap [t,1])$ are strictly increasing, to prove our claim it is enough to check that the equality $\eta(t)=\tilde{\eta}(t)$ holds for all $t\in \mathcal{P}$.

The inequality $\eta(t)\geq\tilde{\eta}(t)$ is straightforward, so we only have to prove $\eta(t)\leq\tilde{\eta}(t)$. If $t_0 \in \mathcal{P}$ then $t_0=.w_0$, with $w_0$ Lyndon; now, if $t_1>t_0$ is in $\mathcal{U}$ then $t_1=.w_1$ with $w_1$ another (possibly infinite) Lyndon word, $w_1>w_0$, and we can define the set 
$$ S:=\{ s=.w_0w: \ \ .w\in K(t_1) \}.$$
It is easy to check that $S \subset \mathcal{U}\cap [t_0,1]$ so $\textup{H.dim }S \leq \textup{H.dim }(\mathcal{U}\cap [t_0,1]).$
Moreover, since $S$ is an affine copy of $ K(t_1)$ we have $\textup{H.dim }S=\textup{H.dim }K(t_1)$.

Thus we have proved that $\tilde{\eta}(t_0) \geq \eta(t_1)$ for all $t_1>t_0$, so our claim follows taking the limit for $t_1\to t_0$ and using the 
continuity of $\eta(t)$.

\end{document}